\documentclass[12pt, a4paper, twoside]{article}

\usepackage[margin=2.5cm]{geometry}
\usepackage{amsmath,amsthm,amssymb,times,mathrsfs,bm}
\usepackage{yhmath}
\usepackage{graphicx}
\usepackage{type1cm}
\usepackage[explicit]{titlesec}
\usepackage{titling}
\usepackage{fancyhdr}
\usepackage{tabularx}
\usepackage[square, comma, numbers, sort&compress]{natbib}
\usepackage[unicode=true, pdfborder={0 0 0}, bookmarksdepth=-1]{hyperref}
\usepackage{faktor}
\usepackage[usenames, dvipsnames]{color}
\usepackage{enumerate}
\usepackage{tikz-cd}
\usepackage{tikz}
\usepackage{bbm}
\usepackage{hyperref}
\usepackage{comment}
\usepackage{sectsty}
\usepackage{parskip}
\usepackage{mathtools}
\usepackage{colonequals}

\AtBeginDocument{
   \def\MR#1{}
}

\sectionfont{\centering\fontsize{15}{15}\selectfont}

\newtheorem{thr}{Theorem}[section]
\newtheorem*{thr*}{Theorem}

\newtheorem*{ex*}{Example}

\newtheorem{df}[thr]{Definition}
\newtheorem*{df*}{Definition}
\newtheorem{pp}[thr]{Proposition}
\newtheorem*{pp*}{Proposition}

\newtheorem*{claim*}{Claim}
\newtheorem{rmk}[thr]{Remark}

\newtheorem*{lemma*}{Lemma}
\newtheorem{cor}[thr]{Corollary}
\newtheorem*{cor*}{Corollary}

\renewcommand\qedsymbol{$\blacksquare$}
\newcommand{\A}{\textbf{A}}
\newcommand{\F}{\textbf{F}}
\newcommand{\K}{\textbf{K}}

\renewcommand{\O}{\mathcal{O}}
\renewcommand{\P}{\textbf{P}}
\newcommand{\Q}{\textbf{Q}}
\newcommand{\D}{\mathfrak{D}}
\newcommand{\R}{\textbf{R}}
\newcommand{\X}{\mathfrak{X}}
\newcommand{\Xg}{X}
\newcommand{\Xs}{X_0}
\newcommand{\Y}{\mathfrak{Y}}

\newcommand{\Z}{\textbf{Z}}

\begin{document}

\setlength{\parindent}{1.5em}

\begin{center}
    \Large
    \textsc{Global Igusa Zeta Functions and {$K$}-Equivalence}
    
    \vspace*{0.4cm} 
    \large
    Shuang-Yen Lee
    
    \vspace*{0.9cm}
\end{center}

\pagestyle{plain}
\thispagestyle{plain}

\setcounter{section}{-1}

\begin{abstract}
    Let $\X$ and $\X'$ be two smooth projective varieties over the ring of integers of a $p$-adic field $\K$ with generic fibers being $X$ and $X'$. We introduce a (family of) good $s$-norms on the pluricanonical spaces of $X$ and $X'$, which are called global Igusa zeta functions in $s$, and show that if the $r$-canonical maps send $X$ and $X'$ birationally to their images respectively, then any isometry between $H^0(X, rK_X)$ and $H^0(X', rK_{X'})$ with respect to this $s$-norm (for some $s > 0$ and $s \neq 1/r$) induces $K$-equivalence on the $\K$-points between $X$ and $X'$. 
\end{abstract}

\section{Introduction}

In Royden's study of isometries on Teichm\"uller spaces, he proved that a compact Riemann surface $M$ of genus $g\ge 2$ is determined by the normed space $(H^0(M, 2K_M), \|{-}\|)$, where 
\[\|\alpha\| = \int_M |\alpha|. \]
Namely, the following Torelli type theorem holds. 

\begin{thr}[{\cite[Theorem 1]{MR0288254}}]
    Let $M$ and $M'$ be compact Riemann surfaces of genus $g\ge 2$, and let 
    \[T\colon \bigl(H^0(M', 2K_{M'}), \|{-}\|\bigr) \longrightarrow \bigl(H^0(M, 2K_{M}), \|{-}\|\bigr) \]
    be a complex linear isometry between the spaces of quadratic differentials. Then there is a conformal map $f\colon M\to M'$ and a complex constant $u$ with $|u| = 1$ such that $T = u\cdot f^*$. 
\end{thr}

We are interested in studying the $p$-adic analogue of Royden's theory. 
Let $\K$ be a fixed $p$-adic (local) field, i.e., a finite extension of $\Q_p$, and $\O$ its ring of integers. Let $\mathfrak{m}$ be the maximal ideal of $\O$, $\pi$ a uniformizer of $\O$, and $\F_{\!q} = \O/\mathfrak{m}$ the residue field of $\O$. For an $n$-dimensional projective variety $\X$ over $\O$, there are two parts of $\X$ --- the generic fiber $X = \X \times_{\operatorname{Spec} \O} \operatorname{Spec} \K$ and the special fiber $\Xs = \X \times_{\operatorname{Spec} \O} \operatorname{Spec} \F_{\!q}$. 

Now assume that $X$ is smooth over $\K$. We can 
view $X(\K)$ as a $\K$-analytic $n$-dimensional manifold (note that we used the smoothness of $X$ here) and thus $X(\K)$ is bianalytic to a finite disjoint union (of copies) of $\O^n$ (see \cite[Sec.~7.5]{MR1743467}). Let us define a quasinorm $\|{-}\|_\K$ on the $r$-pluricanonical space $H^0(X, rK_X)$ as follows. 

\begin{df}
    For any $\alpha \in H^0(X, rK_X)$, we can write it locally as $a(u)\,(du)^r$ on a chart $U\cong \O^n$ with coordinate $u = (u_1, \ldots, u_n)$. We define 
    \[\int_U |\alpha|^{1/r} = \int_{\O^n} |a(u)|^{1/r} d\mu_\O, \]
    where $|{-}|$ is the normalized norm on $\K$ so that $|u| = q^{-v(u)}$, and $\mu_\O$ is the normalized Haar measure on the locally compact abelian group ${\K}^n$ so that $\mu_\O(\O^n) = 1$. Also, we define the $p$-adic norm of $\alpha$ on $X$ by 
        \[\|\alpha\|_{1/r} = \int_{X(\K)} |\alpha|^{1/r} := \sum_U \int_U |\alpha|^{1/r}, \]
        which is independent of the choice of atlas $\{U\}$ by the change of variables formula (see \cite[Sec.~7.4]{MR1743467}). 
\end{df}

Under this analogous $p$-adic setting, in author's bachelor thesis \cite{Leebachelor}, Royden's method was generalized to curves over $\K$ via $p$-adic integrals, and indeed the estimates are easier than those in \cite{MR0288254}. It is natural for us to consider the higher dimensional case. Actually the higher dimensional case over $\textbf{C}$, conjectured by Yau, was generalized by Chi for $r$ being sufficiently large and sufficiently divisible \cite{MR3557304}. 
Indeed, $M$ can be determined up to a birational map by the (pseudo-)norm 
\[\|\alpha\| = \int_M |\alpha|^{2/r}\]
on $H^0(M, rK_M)$. Later on, the same statement for general $r$ was proven by Antonakoudis in \cite[5.2]{MR3251342}. Therefore, we shall do the $p$-adic work by adopting Antonakoudis' approach in the present paper.

Throughout this paper, we will use fraktur font to denote varieties over $\O$ (e.g.~$\mathfrak{X}$, $\mathfrak{X}'$, $\mathfrak{Y}$), normal font to denote their general fibers over $\K$ (e.g.~$X$, $X'$, $Y$), and normal font with subscript $0$ to denote their special fibers over $\textbf{F}_{\!q}$ (e.g.~$X_0$, $X_0'$, $Y_0$). 

From now on, let us assume that $X$ is smooth over $\K$. 
Strictly speaking, we should also require that $X(\K)$ is nonempty so that the integrals are nontrivial. In fact, under a more stringent assumption that $\X$ is smooth over $\O$, there exists a good reduction map
\[h_1\colon X(\K)\longrightarrow \Xs(\F_{\!q}), \]
and $X(\K) \neq \varnothing$ is equivalent to $X_0(\F_{\!q})\neq \varnothing$. 
We shall see that the latter condition can be achieved when $q$ is sufficiently large by using the Weil conjectures (cf. Remark~\ref{rmkWeilbound}). 




The first result in this paper can be stated as follows (cf. Theorem~\ref{thrbir}). 
\begin{thr}\label{thrbir0}
    Let $X$ and $X'$ be smooth projective varieties over $\K$ of dimension $n$, $r$ a positive integer. Suppose there is an isometry
    Then the images of the $\K$-points of the $r$-canonical maps 
    \[\Phi_{|rK_{X}|}\colon X\dashrightarrow \P\bigl(H^0(X, rK_{X})\bigr)^\vee\quad\text{and}\quad \Phi_{|rK_{X'}|}\colon X' \dashrightarrow \P\bigl(H^0(X', rK_{X'})\bigr)^\vee\]
    share a Zariski open dense subset under the identification 
    \[\P(T)^\vee\colon \P\bigl(H^0(X, rK_{X})\bigr)^\vee \xrightarrow{\ \sim\ }\P\bigl(H^0(X', rK_{X'})\bigr)^\vee. \]
\end{thr}

The proof of the complex case given by Antonakoudis was based on Fourier transforms over $\textbf{C}$ used by Rudin in \cite[7.5.2]{MR2446682}. 
However, some properties of Fourier transforms 
used in the proof may not hold in
the \(p\)-adic setting. To achieve the \(p\)-adic statement, we shall use cut-off functions to bypass the difficulty and this in turn simplifies the argument. 




Recall that two birational \(\textbf{Q}\)-Gorenstein
varieties \(X_{1}\) and \(X_{2}\) are \(K\)-equivalent if there exists a common resolution 
\(\phi_{1}\colon Y\to X_{1}\) and \(\phi_{2}\colon Y\to X_{2}\) such that \(\phi_{1}^{\ast}K_{X_{1}} =\phi_{2}^{\ast}K_{X_{2}}\).
For smooth varieties, $K$-equivalence is the same as $c_1$-equivalence. \(K\)-equivalent varieties share many properties;
for example, V.~Batyrev \cite{MR1714818} and C.-L.~Wang \cite{MR1678489} showed that $K$-equivalent smooth varieties over $\textbf{C}$ have the same Betti numbers. From Theorem~\ref{thrbir0}, we obtain a birational map between \(X\) and \(X'\)
and thus it is of interest to know if the birational map indeed gives rise to $K$-equivalence. 

In \cite{MR1714818} and \cite{MR1678489}, the canonical measure $\mu_\X$ 
was intensively used and we
shall recall the construction here
for reader's convenience. By Hensel's lemma, each fiber $h_1^{-1}(\overline{x})$ of the reduction map $h_1$ is $\K$-bianalytic to $\pi \O^n$ (canonically up to a linear transformation, see \cite[Section~10.1]{MR1917232}). Using the canonical measure on $\pi\O^n$ (the normalized Haar measure $\mu_\O$ on the locally compact abelian group $\K^n$ in order that $\mu_\O(\O^n) = 1$), we get a measure on each $h_1^{-1}(\overline{x})$ that can be glued into a measure $\mu_{\X}$ on $X(\K)$. Note that this measure is independent of the choice of the bianalytic map $h_1^{-1}(\overline{x}) \xrightarrow{\sim} \pi \O^n$ and \(\mu_{\X}\) is 
called the canonical measure.

To further characterize \(K\)-equivalence, 
we will construct a new norm on $H^0(X,rK_{X})$ by mixing the canonical measure $d\mu_\X$ and the $p$-adic norm $|\alpha|^{1/r}$;
this is crucial in our approach. For a positive number $s$, the $s$-norm of $\alpha \in H^0(X, rK_X)$ is defined to be 
\[\|\alpha\|_s = \int_{X(\K)} |\alpha|^s\, d\mu_\X^{1 - rs}. \]
It is clear that $\|{-}\|_s$ defines a norm on the space $H^0(X, rK_{X})$. Since the Igusa zeta function associated to an analytic function $f\colon \O^n \to \K$ is defined by 
\[s \xmapsto{\quad} Z(s, f) = \int_{\O^n} |f|^s\,du, \] 
we shall call the assignment
\[\hphantom{\text{$s\in \textbf{C}$, $\operatorname{Re} s \ge 0$,}}s \xmapsto{\quad} \|\alpha\|_s = \int_{X(\K)} |\alpha|^s\, d\mu_{\X}^{1 - rs}, \hspace{10pt} s\in\textbf{C}~\mbox{and}~\operatorname{Re}s\ge 0,\]
the \textit{global Igusa zeta function} associated to $\alpha$. The ``global'' here is to emphasize that they are defined over projective varieties, while the original (local) Igusa zeta functions are defined over the affine space $\O^n$. As in the case of local Igusa zeta functions, global Igusa zeta functions are related to the Poincar\'e series and are rational in $t = q^{-s}$ (see \cite{MR546292}, \cite{MR626951}, \cite{MR751129}). 

We now assume that
\begin{itemize}
    \item[(\hypertarget{assumption}{$\spadesuit$})] \hspace{1cm}
    \centering $\Phi_{|V|}\colon X\dashrightarrow \P(V)^\vee$ and $\Phi_{|V'|}\colon X' \dashrightarrow \P(V')^\vee$ are birational to their images
\end{itemize}
and we are ready to state the second result (cf.~Corollary~\ref{corgI}). 

\begin{thr}
\label{theorem2nd}
    Let $\X$ and $\X'$ be smooth projective varieties over $\O$ of relative dimension $n$ and let $r$ be a positive integer. Suppose the general fibers $X$ and $X'$ satisfy (\hyperlink{assumption}{$\spadesuit$}) with $V = H^0(\Xg, rK_{\Xg})$ and $V' = H^0(\Xg', rK_{\Xg'})$ and there is an isometry
    \[T\colon (V', \|{-}\|_{s}) \longrightarrow (V, \|{-}\|_{s}) \]
    for some positive number $s\neq 1/r$. Then $\Xg$ and $\Xg'$ are $K$-equivalent on the $\K$-points, i.e., there is a smooth projective $Y$ over $\K$ with birational morphisms $\phi\colon Y \to \Xg$ and $\phi'\colon Y \to \Xg'$ such that $\phi^*K_X = {\phi'}^*K_{\Xg'}$ on $Y(\K)$. 
\end{thr}


In \cite{MR1949787}, Kawamata proved that $D$-equivalence (i.e., equivalence between derived categories of coherent sheaves) implies $K$-equivalence for general type smooth projective varieties over $\textbf{C}$. As a $p$-adic normed space analogue, the second result of this paper is to show that isometry on the $p$-adic normed spaces implies $K$-equivalence on the $\K$-points when $X$ and $X'$ satisfy (\hyperlink{assumption}{$\spadesuit$}).

The main idea of the proof is to determine the Jacobian function by the norms. More precisely, for a resolution $\phi\colon Y\to X$, we use the norms on $H^0(X, rK_X)$ and $H^0(Y, rK_{Y})$ to determine $J_\phi = \phi^*d\mu_{\X} / d\mu_Y$, based on the proof of Theorem~\ref{thrbir0}. 

The same proof applies to the case where $\X$ is replaced by a klt pair $(\X, \D)$ so that the canonical measure $d\mu_{\X, \D}$ (defined in Section~\ref{secgI}) is finite and the space $H^0(X, rK_X)$ is replaced by a linear subspace $V$ of $H^0(X, \lfloor r(K_X + D + L)\rfloor)$ for some $\R$-divisor $L$ in order that the map $\Phi_{|V|}$ maps $X$ birationally into its image. 

I have to remark that Chi posted a similar result to Theorem~\ref{thrbir0} in \cite{arxiv.2211.09335} later on. In fact, he was in my bachelor thesis oral examination committee although he resigned after reviewing my manuscript. The paper is organized as follows. In Section~\ref{secgI}, we will construct the global Igusa zeta function using the \(p\)-adic norm and the canonical measure. In Section~\ref{secbir}, we will give a proof of Theorem \ref{thrbir0}. In Section~\ref{secKequiv}, we will prove Theorem \ref{theorem2nd}, our 
result on \(K\)-equivalence. 

\section*{Acknowledgement}

I would like to thank Professor C.-L.~Wang for introducing this subject to me and suggesting me mix the original $p$-adic norm and the canonical measure. I would also like to thank T.-J.~Lee, H.-W.~Lin, P.-H.~Chuang, S.-S.~Wang, H.-Y.~Yao, and J.-D.~Yu for useful discussions or comments. 

\section{Global Igusa zeta functions}\label{secgI}

As in the introduction, for a $p$-adic (local) field $(\K, |{-}|)$, let 
\begin{itemize}
    \item $\O = \{x\in \K\mid |x| \le 1\}$ be the ring of integers, 
    \item $\mathfrak{m} = \{x\in \K\mid |x| < 1\} = (\pi)$ the maximal ideal of $\O$, 
    \item $v\colon\K \to \Z\cup\{\infty\}$ the valuation on $\K$, and
    \item $\F_{\!q} = \faktor{\O}{\mathfrak{m}}$ the residue field. 
\end{itemize}
Here we scale the norm in order that $|u| = q^{-v(u)}$ for all $u\in \K^\times$. 

Fix a positive integer $r$. Let $\X$ be an $n$-dimensional projective variety over $\O$ with the following assumptions: its general fiber $X$ is smooth over $\K$, and $X(\K)$ (as a $\K$-analytic manifold) is nonempty. It follows from the valuative criterion for properness that $\X(\O) = \X(\K) = X(\K)$. 

We have defined the norm $\|{-}\|_{1/r}$ on the space of pluricanonical $r$-forms $H^0(\Xg, rK_{\Xg})$ via the $p$-adic integral: 
\[\|\alpha\|_{1/r} = \int_{\Xg(\K)} |\alpha|^{1/r}. \]
The assumption that $\Xg(\K)$ being nonempty is to make sure that the norm is not trivial (by implicit function theorem). Another reason to make this assumption is that we want $\Xg(\K)$ to be Zariski dense in $\Xg$. Indeed, $\Xg(\K)$ contains a $\K$-analytic open subset that is bianalytic to an $n$-dimensional polydisc $\Delta$ with the result that the dimension of the Zariski closure of $\Xg(\K)$ is at least $n$. 

In order to resolve the non-general type case, we shall extend the definition of the norm to a Kawamata log-terminal (klt for short) pair $(X, D)$: 
\[\|\alpha\|_{1/r, D} \colonequals \int_{\Xg(\K)} |\alpha|^{1/r},\quad \alpha \in H^0(\Xg, \lfloor r(K_X + D)\rfloor ). \]

When $\X$ is smooth over $\O$, we have defined for each positive number $s$ the $s$-norm 
\[\|\alpha\|_s = \int_{\Xg(\K)} |\alpha|^s \,d\mu_{\X}^{1-rs}, \]
where $d\mu_{\X}$ is the canonical measure on $\Xg(\K)$ defined by the mod-$\mathfrak{m}$ reduction 
\[h_1\colon \X(\O) \longrightarrow \Xs(\F_{\!q}), \]
where $\Xs$ is the special fiber of $\X$ over $\F_{\!q}$. 

For a klt log pair $(\X, \D)$, we can consider the measure $d\mu_{\X, \D}$, which is locally the $p$-adic norm $|\omega|^{1/r_\D}$ of a generator $\omega\in H^0(\mathfrak{U}, r_\D(K_{\X} + \D))(\O)$, where $r_\D$ is a positive integer such that $r_\D(K_{\X} + \D)$ is Cartier and $\mathfrak{U}\subseteq \X$ is an open subset on which $r_\D(K_{\X} + \D)$ is free. Note that this $p$-adic norm does not depend on the choice of $\omega$, $r_\D$ and $\mathfrak{U}$, and hence glues into $d\mu_{\X, \D}$ on $X(\K)$. 

For any $\R$-divisor $L$ and any positive integer $r$, we can define $\|{-}\|_{s, \D}$ similarly on the space $H^0(\Xg, \left\lfloor r(K_X + D + L)\right\rfloor)$ by
\[\|\alpha\|_{s, \D} = \int_{\Xg(\K)} |\alpha|^s\, d\mu_{\X, \D}^{1 - rs}\]
when $\operatorname{Re} s \ge 0$ is small enough. 

Indeed, replacing $L$ by $L'$, which is determined by the equation
\[r(K_X + D + L') = \lfloor r(K_X + D + L) \rfloor, \]
we may assume that $r(K_X + D + L)$ is Cartier. Let $F$ be the fixed locus divisor of the linear system $|r(K_X + D + L)|$ and consider a log-resolution $\phi\colon Y \to \Xg$ such that 
\[\phi^*D = \sum_{E\in \mathcal{E}} d_E E,\quad \phi^*L = \sum_{E\in \mathcal{E}} \ell_E E,\quad \phi^*F = \sum_{E\in\mathcal{E}} f_E E,\quad K_{Y} = \phi^*K_X + \sum_{E\in \mathcal{E}} m_E E \]
with $\sum E$ a smooth normal crossing divisor. Then $\|{-}\|_{s, \D}$ is defined on $H^0(\Xg, r(K_X + D + L))$ when 
\[\operatorname{Re}\bigl(s \cdot (f_E + r(m_E - d_E - \ell_E)) + (1 - rs) \cdot (m_E - d_E)\bigr) > -1, \quad \forall E\in \mathcal{E}, \]
which is equivalent to 
\[\operatorname{Re} s < s_r(\X, \D, L) \colonequals \inf_{E \in \mathcal{E}^+} \left(\frac{m_E - d_E + 1}{r\ell_E - f_E}\right), \] 
where $\mathcal{E}^+ = \{E\in \mathcal{E}\mid r\ell_E > f_E\}$. Note that the klt assumption on the log pair $(\X, \D)$ is used here in order that $s_r(\X, \D, L) > 0$. 

Since $s_r(\X, \D, L)$ is defined to be the largest number such that the integral of $|\alpha|^s\,d\mu_{\X, \D}^{1-rs}$ is finite for all $0 \le \operatorname{Re} s < s_r(\X, \D, L)$ and $\alpha \in H^0(\Xg, r(K_X + D + L))$, it is independent of the choice of the resolution $\phi$. 

For general $L$ and a linear subspace $V$ of $H^0(\Xg, \left\lfloor r(K_X + D + L)\right\rfloor)$, $s_r(\X, \D, L)$ is defined to be $s_r(\X, \D, L')$, where $L'$ is determined by the equation
\[r(K_X + D + L') = \lfloor r(K_X + D + L) \rfloor; \]
$s_r(\X, \D, V) \ge s_r(\X, \D, L)$ is defined to be the largest number such that the integral of $|\alpha|^s\,d\mu_{\X, \D}^{1-rs}$ is finite for all $0 \le \operatorname{Re} s < s_r(\X, \D, V)$ and $\alpha \in V$. 

In this (smooth) case, the assumption on $\Xg(\K)$ being nonempty is equivalent to $\Xs(\F_{\!q})$ being nonempty. 
\begin{rmk}\label{rmkWeilbound}
    Using the Weil conjectures \cite{MR340258}, we see that  
    \[\#\Xs(\F_{\!q}) = \sum_{i = 0}^{2n} \sum_{j = 1}^{h_i} (-1)^i \alpha_{ij} \ge q^{n} + 1 - \sum_{i = 1}^{2n-1} h^i q^{i/2}, \]
    for some algebraic integers $\alpha_{ij}$ with $|\alpha_{ij}| = q^{i/2}$, where $h^i = h^i(\Xg)$ are the Betti numbers. Hence, the condition $\# \Xs(\F_{\!q}) > 0$ could be achieved by some bounds on $q$ and $h^i$'s. For example, 
    \[q \ge \left(\sum_{i=1}^{2n - 1}h^i\right)^2 \]
    will do. We denote by $q_0(\X)$ to be the smallest positive integer such that 
    \[q^n + 1 \ge \sum_{i = 1}^{2n-1} h^i q^{i/2},\quad \forall q \ge q_0(\X). \]
    In particular, when $\X$ is a curve (over $\O$) of genus $g \ge 1$, $q_0(\X) = 4g^2 - 2$. 
\end{rmk}

For each $k\ge 1$, consider the mod-$\mathfrak{m}^k$ reductions 
\[\begin{tikzcd}[column sep = 0pt, row sep = 0pt]
    \X(\O) \ar[rr, "h_k"] &~~& \X(\O/\mathfrak{m}^k)\\
    x \ar[rr, mapsto] &~~& \overline{x}^{(k)}
\end{tikzcd}\quad\text{and}\quad\begin{tikzcd}[column sep = 0pt, row sep = 0pt]
    H^0(\X, rK_{\X})(\O) \ar[rr] &~~& H^0(\X, rK_{\X})(\O/\mathfrak{m}^k)\\
    \alpha \ar[rr, mapsto] &~~& \overline{\alpha}^{(k)}, 
\end{tikzcd}\]
where $H^0(\X, rK_{\X})(R)$ denotes the $R$-points of $H^0(\X, rK_{\X})$ for $R = \O$, $\O/\mathfrak{m}^k$. Applying the method in \cite{MR441933}, we may calculate the norm by the numbers of the zeros of $\alpha \in H^0(\X, rK_{\X})(\O)$ on each $\X(\O/\mathfrak{m}^k)$: 

\begin{pp}\label{ppcal}
    For a nonzero element $\alpha \in H^0(\X, rK_{\X})(\O)$, we have 
    \[\|\alpha\|_s = \frac{\# \Xs(\F_{\!q})}{q^n} - (q^s - 1) \sum_{k=1}^\infty \frac{N_k}{q^{k(n + s)}}, \]
    where 
    \[N_k = \#\left\{\overline{x}^{(k)}\in \X(\O/\mathfrak{m}^k)\,\middle|\, \overline{\alpha}^{(k)}\bigl(\overline{x}^{(k)}\bigr) = 0\right\} \]
    is the cardinality of the zero set of $\overline{\alpha}^{(k)}\in H^0(\X, rK_{\X})(\O/\mathfrak{m}^k)$ on $\X(\O/\mathfrak{m}^k)$. 
\end{pp}

\begin{proof}
    Since $\X$ is smooth over $\O$, each fiber of $h_1$ is $\K$-bianalytic to $\pi \O^n$ with measure preserved. For $\overline{x}^{(1)}\in \Xs(\F_{\!q})$, let $u = (u_1, \ldots, u_n)$ be a local coordinate of $h_1^{-1}(\overline{x}^{(1)})$. Then 
    \[\int_{h_1^{-1}(\overline{x}^{(1)})} |\alpha|^s = \int_{\pi \O^n} |a(u)|^s\, du\]
    for some analytic function $a(u) \in \O[[u]]$. Let 
    \[N_k\bigl(\overline{x}^{(1)}\bigr) = \#\left\{\overline{y}^{(k)} \in h_k(h_1^{-1}(\overline{x}^{(1)}))\,\left|\, \overline{\alpha}^{(k)}\bigl(\overline{y}^{(k)}\bigr) =0\right.\right\} \]
    be the cardinality of the zero set of $\overline{\alpha}^{(k)}$ on the fiber of $\overline{x}^{(1)}$. Then it follows from 
    \[\mu_\O\left(|a|^{-1}\left([0, q^{-k}]\right)\right) = \mu_\O\left(\left\{u\bigm| \overline{a}^{(k)}\bigl(\overline{u}^{(k)}\bigr) = 0\right\}\right) = \frac{N_k\bigl(\overline{x}^{(1)}\bigr)}{q^{kn}}\]
    that
    \begin{align*}
        \int_{\pi \O^n} |a(u)|^s\,du &= \sum_{k = 0}^\infty q^{-ks}\cdot \mu_\O(|a|^{-1}(q^{-k})) = \frac{1}{q^n} - (q^s - 1)\sum_{k = 1}^\infty q^{-ks}\cdot \frac{N_k\bigl(\overline{x}^{(1)}\bigr)}{q^{kn}}. 
    \end{align*}
    Summing the above equation over $\overline{x}^{(1)} \in \Xs(\F_{\!q})$, we get 
    \begin{align*}
        \|\alpha\|_s &= \sum_{\ \ \overline{x}^{(1)}}\left(\frac{1}{q^n} - (q^s - 1)\sum_{k = 1}^\infty \frac{N_k\bigl(\overline{x}^{(1)}\bigr)}{q^{k(n+s)}}\right) = \frac{\# \Xs(\F_{\!q})}{q^n} - (q^s - 1) \sum_{k=1}^\infty \frac{N_k}{q^{k(n + s)}}. \qedhere
    \end{align*}
\end{proof}

In particular, we have \cite[Theorem~2.2.5]{MR670072}
\[\|\alpha\|_0 = \int_{\Xg(\K)} \, d\mu_{\X} = \frac{\# \Xs(\F_{\!q})}{q^n}, \quad \|\alpha\|_\infty \colonequals \lim_{s \to \infty} \|\alpha\|_s = \frac{\# \Xs(\F_{\!q}) - N_1}{q^n}. \]
We see that 
\[\|\alpha\|_s = \frac{\# \Xs(\F_{\!q})}{q^n} - (1 - t) \sum_{k = 1}^\infty \frac{N_{k}}{q^{kn}} t^{k - 1} \]
is a holomorphic function in $t = q^{-s}$ near $t = 0$. As (local) Igusa zeta functions are rational functions in $t$ \cite{MR546292}, \cite{MR626951}, the similar result also holds for global Igusa zeta functions: 

\begin{pp} 
    The holomorphic function $t \mapsto \|\alpha\|_{s}$ is a rational function in $t = q^{-s}$ of the form 
    \[\frac{P_\alpha(t)}{\prod_E (q^{m_E+1} - t^{a_E})}, \]
    where $P_\alpha$ is a polynomial with coefficients in $\Z[q^{-1}]$ and $\{(a_E, m_E)\}_{E\in \mathcal{E}}$ are the 
    discrepancies associated to the divisor $(\alpha)$. 
\end{pp}

In fact, it holds for the log pair case: 

\begin{pp}
    Let $L$ be an $\R$-divisor on a smooth projective klt pair $(\X, \D)$. For a nonzero element $\alpha \in H^0(\Xg, \lfloor r(K_X + D + L)\rfloor)$, the function $t \mapsto \|\alpha\|_{s, \D}$ is a rational function in $t_\D = q^{-s/r_\D}$ of the form 
    \[\frac{P_\alpha(t_\D)}{t_\D^{e} \prod_E \bigl(q^{m_E - d_E +1} - t_\D^{r_\D(a_E - rd_E)}\bigr)}, \]
    where $r_\D$ is the least positive integer such that $r_\D(K_X + D)$ is Cartier, $P_\alpha$ is a polynomial with coefficients in $\Z[q^{-1}]$, $e$ is a nonnegative integer and $\{(a_E, d_E, m_E)\}_{E\in \mathcal{E}}$ are the discrepancies associated to the divisors $(\alpha)$ and $D$. 
\end{pp}

For the original ($\D = 0$) case, $r_\D = 1$, and notice that $\|\alpha\|_{s}$ is bounded as $t$ tends to $0$, so the exponent $e$ could be chosen to be $0$. 

\begin{proof}
    The proof is essentially the same as in \cite{MR626951}, using resolution of singularities. Consider a log resolution $\phi\colon Y\to \Xg$ such that 
    \[\phi^*(\alpha) = \sum_{E} a_E E,\quad \phi^*D = \sum_E d_E E,\quad K_{Y} = \phi^*K_X + \sum_{E} m_E E \]
    with $\sum E$ a normal crossing divisor. 
    
    It follows from the definition that 
    \[\|\alpha\|_{s, \D} = \int_{\Xg(\K)} |\alpha|^{s}\, d\mu_{\X, \D}^{1 -rs} = \int_{Y(\K)} |\phi^*\alpha|^s \cdot \phi^* d\mu_{\X, \D}^{1 - rs}. \]
    Decompose $Y(\K)$ into disjoint compact charts $Y_i$ such that in coordinate we have
    \[\phi^*\alpha = a_i(u)\cdot u^{A_i + rM_i}\, (du)^{\otimes r}, \quad \phi^*d\mu_{\X, \D}(u) = m_i(u)\cdot |u|^{M_i - D_i}\, du, \]
    where $a_i(u) \neq 0$ and $m_i(u) \neq 0$ for all $u \in U_i$. After further decomposing $Y_i$, we may assume that each $Y_i$ is a polydisc and that $|a_i(u)| \equiv q^{-a}$ and $m_i(u)\equiv q^{-m}$ are constants on $Y_i$. Then 
    \[\int_{Y(\K)} |\phi^*\alpha|^{s} \cdot \phi^*d\mu_{\X, \D}^{1 - rs} = \sum_{i} \int_{Y_i} q^{-as}|u|^{sA_i} \cdot q^{-m(1 - rs)}|u|^{M_i - (1 - rs)D_i}\, du. \]
    
    Say $Y_i$ is the polydisc $\pi^{k} \O^n$. It is easy to calculate that 
    \begin{align*}
        \int_{\pi^k \O^n} &q^{-as}|u|^{sA_i}\cdot q^{-m(1 - rs)} |u|^{M_i - (1 - rs)D_i}\, du \\
        &= q^{-m} t_D^{r_D(a - rm)} \prod_j \int_{\pi^k \O^n} |u_j|^{sa_{i,j} + m_{i,j} - (1-rs)d_{i,j}} du \\
        &= q^{-m} t_D^{r_D(a - rm)} \prod_{j} \frac{(1 - q^{-1})q^{-k(sa_{i,j} + m_{i,j} - (1-rs)d_{i,j} + 1)}}{1 - q^{-(sa_{i,j} + m_{i,j} - (1-rs)d_{i,j} + 1)}} \\
        &= \frac{(1 - q^{-1})^n}{q^{m + (k-1)(\sum (m_{i,j} - d_{i,j} + 1))}} \cdot \frac{t_D^{r_D(a - rm + k\sum (a_{i,j} - rd_{i,j}))}}{\prod_{j} (q^{m_{i,j} - d_{i,j} + 1} - t^{r_D(a_{i,j} - rd_{i,j})})}. 
    \end{align*}
    Summing over $i$, we see that $\|\alpha\|_{s, \D}$ is a rational function in $t_D$ of the form
    \[\frac{P_\alpha(t_\D)}{t_\D^{e} \prod_E (q^{m_E - d_E + 1} - t_\D^{r_\D(a_E - rd_E)})}, \]
    where $P_\alpha$ is a polynomial with coefficients lie in $\Z[q^{-1}]$ and $e$ is a nonnegative integer (coming from the term $t_\D^{r_\D(a - rm + k\sum (a_{i,j} - rd_{i,j}))}$), as desired. 
\end{proof}

This proposition allows us to view $t_\D = q^{-s/r_\D}$ as a formal variable, and view the total norm $\|\alpha\|_{\D} \colonequals (\|\alpha\|_{s, \D})_{s}$ of a form $\alpha$ as an element in the function field $\Q(t_\D)$. Hence, there is a $\Q(t_\D)$-valued norm
\[\|{-}\|_{\D}\colon \varinjlim_{L} H^0(\Xg, \left\lfloor r(K_X + D + L) \right\rfloor) \longrightarrow \Q(t_\D). \]
As a consequence, when $t_{\D,0} = q^{-s_0/r_\D}$ is a transcendental number for some $s_0$, we can determine the rational function $\|\alpha\|_{\D}$ by the value $\|\alpha\|_{s_0, \D}$. So the norms $\|{-}\|_{s, \D}$ on the space $H^0(\Xg, \lfloor r(K_X + D + L) \rfloor)$, $0 \le \operatorname{Re} s < s_r(\X, \D, L)$, are in fact determined by $\|{-}\|_{s_0, \D}$. 

\section{Characterizing birational models}\label{secbir}

Suppose that $f\colon X \dashrightarrow X'$ is a birational map between smooth projective varieties over $\K$. Then the properness of $X$ shows that there is an open set $U \subset X$ on which $f|_U\colon U\to X'$ is a birational morphism with $\operatorname{codim}_{X}(X \setminus U) \ge 2$. It follows that
\[f^*\colon \bigl(H^0(X', rK_{X'}), \|{-}\|_{1/r}\bigr) \longrightarrow \bigl(H^0(U, rK_{X}|_U),\|{-}\|_{1/r}\bigr) \cong \bigl(H^0(X, rK_{X}),\|{-}\|_{1/r}\bigr)\]
is an isometry by change of variables. This means that the normed space 
\[\bigl(H^0(X, rK_{X}), \|{-}\|_{1/r}\bigr)\]
only reflects the birational class of $X$. 
For simplicity, let us denote by $V_{r, X}$ the space $H^0(X, rK_{X})$ and by $\Phi_{r, X}$ the $r$-canonical map 
\[\Phi_{|rK_{X}|}\colon X \dashrightarrow \P\bigl(H^0(X, rK_{X})\bigr)^\vee. \] 
For a klt log pair $(X, D)$ and an $\R$-divisor $L$ on $X$, denote by $V_{r, (X, D), L}$ the space 
\[H^0(X, \left\lfloor r(K_{X} + D + L)\right\rfloor)\]
and by $\Phi_{r, (X, D), L}$ the map 
\[\Phi_{|\left\lfloor r(K_{X} + D + L)\right\rfloor|}\colon X \dashrightarrow \P(V_{r, (X, D), L})^\vee. \] 
For a linear subspace of $V$ of $V_{r, (X, D), L}$, denote by $\Phi_V$ the map determined by the linear system $|V|$. 

Using the $p$-adic analogue of the trick in \cite[Theorem~5.2]{MR3251342}, we can prove that: 

\begin{thr}\label{thrbir}
    Let $X$ and $X'$ be smooth projective varieties over $\K$ of dimension $n$ with nonempty $\K$-points, $r$ a positive integer. Suppose there is a ($\K$-linear) isometry
    \[T\colon \bigl(V_{r, \Xg'}, \|{-}\|_{1/r}\bigr) \longrightarrow \bigl(V_{r, \Xg}, \|{-}\|_{1/r}\bigr). \]
    Then the images of the $r$-canonical maps $\Phi_{r, \Xg}$ and $\Phi_{r, \Xg'}$ are birational to each other. 
    
    When $\X$ and $\X'$ are both smooth over $\O$, the statement also holds when the norm $\|{-}\|_{1/r}$ is replaced by $\|{-}\|_s$ for any positive number $s$. 
\end{thr}

\begin{proof}
    Let $\alpha_0$, $\alpha_1$, $\ldots$ , $\alpha_N$ be a basis of $V_{r, \Xg}$, and let $\alpha_i' = T^{-1}(\alpha_i)$, which is a basis of $V_{r, \Xg'}$. Then the $r$-canonical maps $\Phi_{r,\Xg}$ and $\Phi_{r, \Xg'}$ could be realized as 
    \begin{align*}
        [\alpha_0: \cdots : \alpha_N]\colon \Xg \dashrightarrow \P^N\quad\text{and}\quad [\alpha_0': \cdots : \alpha_N']\colon \Xg' \dashrightarrow \P^N, 
    \end{align*}
    respectively. 
    
    In the following, $s = 1/r$ if $\Xg$ and $\Xg'$ are just smooth projective varieties over $\K$. 
    
    Let $d\nu = |\alpha_0|^{s}\,d\mu_\X^{1-rs}$ and $g_i = \frac{\alpha_i}{\alpha_0}$ on $\Xg^\circ = \Xg\setminus\{\alpha_0 = 0\}$, $d\nu' = |\alpha_0'|^{s}\,d\mu_{\X'}^{1-rs}$ and $g_i' = \frac{\alpha_i'}{\alpha_0'}$ on ${\Xg'}^\circ = \Xg'\setminus\{\alpha_0' = 0\}$. Using the isometry $T$, we get 
    \begin{align}\label{eqmeas}
        \int_{{\Xg}^\circ(\K)} \Bigl|1 + \sum_{i}\lambda_i g_i\Bigr|^s\, d\nu &= \Bigl\|\alpha_0 + \sum_{i}\lambda_i \alpha_i\Bigr\|_s \notag\\
        &= \Bigl\|\alpha_0' + \sum_{i}\lambda_i \alpha_i'\Bigr\|_s = \int_{{X'}^\circ(\K)} \Bigl|1 + \sum_{i}\lambda_i g_i'\Bigr|^s\, d\nu'
    \end{align}
    for all $(\lambda_1, \ldots, \lambda_N)\in \K^N$. 
    
    Now, we need a $p$-adic analogue of Rudin's result \cite[7.5.2]{MR2446682}: 
    
    \noindent{\bf Claim.} We have
    \begin{equation}\label{eqRudin}
        \int_{X^\circ} h\circ (g_1, \ldots, g_N)\,d\nu = \int_{{X'}^\circ} h\circ (g_1', \ldots, g_N')\,d\nu'
    \end{equation}
    for all nonnegative Borel function $h\colon \K^N \to \R$. 
    
    \begin{proof}[Proof of Claim]\renewcommand{\qedsymbol}{$\square$}
        Let $W$ be the set of all Borel function $h$ such that (\ref{eqRudin}) holds. 
        It follows that $W$ is invariant under translations, dilations, and scaling. So it suffices to prove that $\boldsymbol{1}_{\O^N}\in W$. 
        
        Let $G = (g_1, \ldots, g_N)\colon {X}^\circ \to \A^{\!N}$ and $G' = (g_1', \ldots, g_N')\colon {X'}^\circ \to \A^{\!N}$. Define 
        \[B_s(y) = B_s(y_1, \ldots, y_N) = \int_{\O^N} \Bigl|1 + \sum_{i} \lambda_i y_i\Bigr|^{s}\,d\lambda, \]
        where $y = (y_1, \ldots, y_N) \in \K^N$ and $d\lambda = d\lambda_1\cdots d\lambda_N$. By (\ref{eqmeas}) and Fubini's theorem, 
        \begin{align*}
            \int_{X(\K)} B_s\circ G\,d\nu &= \int_{\O^N} \int_{{X}^\circ(\K)} \Bigl|1 + \sum_{i}\lambda_i g_i\Bigr|^{s}\, d\nu d\lambda \\
            &= \int_{\O^N} \int_{{X'}^\circ(\K)} \Bigl|1 + \sum_{i}\lambda_i g_i'\Bigr|^{s}\, d\nu' d\lambda = \int_{X'(\K)} B_s\circ G'\,d\nu'. 
        \end{align*}
        By change of variables, we see that 
        \begin{equation}
            B_s(y) = \begin{cases}
                1, & \text{ if }\max |y_i| < 1, \\[-6pt]
                b(s)\cdot \max |y_i|^{s}, & \text{ if }\max |y_i| \ge 1, 
            \end{cases} \label{eqBs}
        \end{equation}
        where $b(s) = \int_{\O} |y|^{s}\,dy = \frac{q - 1}{q - t} \in (0, 1)$ and $t = q^{-s}$. So, if we define the ``cut-off'' function
        \begin{equation}
            B_{s, 0}(y) = t^{-1} B_s(\pi y) - B_s(y) = \begin{cases}
                t^{-1} - 1, & \text{ if }\max |y_i| < 1, \\[-6pt]
                t^{-1} - b(s), & \text{ if }\max |y_i| = 1, \\[-6pt]
                0, & \text{ if }\max |y_i| > 1, 
            \end{cases} \label{eqBs0}
        \end{equation}
        then $B_{s, 0}(-)$ also satisfies (\ref{eqRudin}), and hence lies in $W$. Since $B_{s, 0}(-)$ is a Schwartz–Bruhat function that is supported in $\O^N$ with $\int_{\O^N} B_{s, 0}(y)\,dy\neq 0$ and well-defined on $\O^N/\pi\O^N$, we see that 
        \begin{equation}\label{eqcharON}
            \boldsymbol{1}_{\O^N}(y) = \left(\frac{1}{\mu_\O(\pi\O^N)}\int_{\O^N} B_{s, 0}(u)\,du\right)^{-1}\sum_{\overline{z}\in \O^N/\pi\O^N} B_{s, 0}(y + z) 
        \end{equation}
        also lies in $W$, as desired. \qedhere
    \end{proof}
    
    Taking $h = \boldsymbol{1}_{G(X^\circ)}$ in the claim, we see that
    \[\|\alpha_0\|_{s} = \int_{G^{-1}(G(X^\circ))}|\alpha_0|^{s}\,d\mu_X^{1 - rs} = \int_{{G'}^{-1}(G(X^\circ))}|\alpha_0'|^{s}\, d\mu_{X'}^{1 - rs} \le \|\alpha_0'\|_{s} = \|\alpha_0\|_{s}. \]
    So the inequality above has to be an equality. Let $U$ be the intersection of $G(X^\circ)$ and $G'({X'}^\circ)$. The equality implies that the $\K$-points $U(\K)$ of $U$ has full measure in $G'({X'}^\circ)(\K)$ with respect to $G_*'\nu'$. 
    
    Let $\bar{X}$ (resp.~$\bar{X}'$) be the image of $\Phi_{r, X}$ (resp.~$\Phi_{r, X'}$), and let $\bar{n}$ (resp.~$\bar{n}'$) be the dimension of $\bar{X}$ (resp.~$\bar{X}'$). Since the general fiber of $X\dashrightarrow \bar{X}$ has dimension $n - \bar{n}$, the image of an $n$-dimensional polydisc $\Delta$ in $X(\K)$ under $X \dashrightarrow \bar{X}$ has dimension $\bar{n}$ (at a generic point). This argument also holds for $X'\dashrightarrow \bar{X}'$. So we see that $\bar{n} = \bar{n}'$ and that $U(\K)$ contains an $\bar{n}$-dimensional polydisc. 
    
    Since $U_{\K}$ is now a quasi-projective variety of dimension at least $\bar{n}$ (and hence equal to $\bar{n}$), the images $\bar{X}$ and $\bar{X}'$ are birational to each other. 
\end{proof}

\begin{cor}\label{corbir}
    Let $\X$ and $\X'$ be smooth projective varieties over $\O$ of relative dimension $n$, $r$ a positive integer. Suppose that $q \ge \max\{q_0(\X), q_0(\X')\}$, the $r$-canonical maps $\Phi_{r, X}$ and $\Phi_{r, X'}$ maps $X$ and $X'$ birationally to their images respectively, and there is an isometry 
    \[T\colon \bigl(V_{r, X'}, \|{-}\|_{s}\bigr) \longrightarrow \bigl(V_{r, X}, \|{-}\|_{s}\bigr) \]
    for some positive number $s$. Then there is a birational map $f\colon X \dashrightarrow X'$ such that $T = u\cdot f^*$ for some $u \in \K^\times$.  
\end{cor}

\begin{proof}
    The condition $q \ge \max\{q_0(\X), q_0(\X')\}$ ensures that both $X(\K)$ and $X'(\K)$ are nonempty. Since the images $\bar{X}$ and $\bar{X}'$ are birational to each other, $X$ and $X'$ are birational to each other. The birational map $f\colon \X\dashrightarrow \X'$ comes from the identification 
    \[\P(T)^\vee\colon \P(V_{r, X})^\vee \xrightarrow{\ \sim\ } \P(V_{r, X'})^\vee. \]
    Therefore, $T = u\cdot f^*$ for some $u\in \K^\times$. 
\end{proof}

Imitating the proof, Theorem~\ref{thrbir} and Corollary~\ref{corbir} are also generalized to the log pair case (with only difference in notation): 

\begin{thr}
    Let $(X, D)$ and $(X', D')$ be smooth projective klt pairs over $\K$ of dimension $n$ with nonempty $\K$-points, $r$ a positive integer. Let $L$ (resp.~$L'$) be an $\R$-divisor on $X$ (resp.~$X'$), and let $V$ (resp.~$V'$) be a linear subspace of $V_{r, (X, D), L}$ (resp.~$V_{r, (X', D'), L'}$). Suppose $1/r < \min\{s_r(X, D, V), s_r(X', D', V')\}$ and there is an isometry
    \[T\colon \bigl(V', \|{-}\|_{1/r, D'}\bigr) \longrightarrow \bigl(V, \|{-}\|_{1/r, D}\bigr). \]
    Then the images of the maps $\Phi_{V}$ and $\Phi_{V'}$ are birational to each other. 
    
    When $\X$ and $\X'$ are both smooth over $\O$, the statement also holds when $1/r$ is replaced by any positive number $s < \min\{s_r(\X, \D, V), s_r(\X', \D', V')\}$. 
\end{thr}

\begin{cor}
    Let $(\X, \D)$ and $(\X', \D')$ be smooth projective klt pairs over $\O$ of relative dimension $n$, $r$ a positive integer. Let $L$ (resp.~$L'$) be an $\R$-divisor on $X$ (resp.~$X'$), and let $V$ (resp.~$V'$) be a linear subspace of $V_{r, (X, D), L}$ (resp.~$V_{r, (X', D'), L'}$). 
    
    Suppose that $q \ge \max\{q_0(\X), q_0(\X')\}$, the maps $\Phi_{V}$ and $\Phi_{V'}$ maps $X$ and $X'$ birationally to their images, and there is an isometry 
    \[T\colon \bigl(V', \|{-}\|_{s, \D'}\bigr) \longrightarrow \bigl(V, \|{-}\|_{s, \D}\bigr) \]
    for some positive number $s < \min \{s_r(\X, \D, V), s_r(\X', \D', V')\}$. Then there is a birational map $f\colon \X \dashrightarrow \X'$ such that $T = u\cdot f^*$ for some $u \in \K^\times$.  
\end{cor}

\section{Characterizing the $K$-equivalence}\label{secKequiv}

We say two projective smooth varieties $\X$ and $\X'$ over $\O$ (resp.~$\Xg$ and $\Xg'$ over $\K$) are $K$-equivalent on the $\K$-points if
there is a projective smooth variety $\Y$ over $\O$ (resp.~$Y$ over $\K$) with birational morphisms $\phi\colon \Y \to \X$ and $\phi'\colon \Y \to \X'$ (resp.~$\phi\colon Y\to X$ and $\phi'\colon Y\to X'$) such that $\phi^*K_X = {\phi'}^*K_{X'}$ on $Y(\K)$. Equivalently, 
\[K_Y = \phi^*K_X + \sum m_E E = {\phi'}^*K_{X'} + \sum m_E E \]
on $Y(\K)$. For $K$-equivalence between klt pairs $(\X, \D)$ and $(\X', \D')$, simply replace the relation $\phi^*K_X = {\phi'}^*K_{X'}$ by $\phi^*(K_X + D) = {\phi'}^*(K_{X'} + D')$. 

When $\X$ and $\X'$ are $K$-equivalent on the $\K$-points over $\O$, the pullbacks of the canonical measures $\phi^*d\mu_{\X}$ and ${\phi'}^*d\mu_{\X'}$ are equal to each other. Hence, the natural linear transformation 
\begin{equation}\label{eqisoKequiv} 
    \bigl(V_{r, X'}, \|{-}\|_s\bigr)\longrightarrow \bigl(V_{r, X}, \|{-}\|_s\bigr)
\end{equation}
is an isometry:
\[ \|\alpha'\|_s = \int_{Y(\K)} |{\phi'}^*\alpha'|^s\cdot {\phi'}^*d\mu_{\X'}^{1-rs} = \int_{Y(\K)} |\phi^*\alpha|^s\cdot \phi^*d\mu_{\X}^{1-rs} = \|\alpha\|_s. \]
The main result of this section is to show that the isometry (\ref{eqisoKequiv}) gives us $K$-equivalence on the $\K$-points between $\Xg$ and $\Xg'$ when $\Xg$ and $\Xg'$ satisfy (\hyperlink{assumption}{$\spadesuit$}). 

In order to prove this, let us consider a proper birational morphism $\phi\colon Y\to X$ over $\K$ such that 
\[K_Y = \phi^* K_X + \sum_{E\in \mathcal{E}} m_E E \]
on $Y(\K)$. Suppose that the $r$-canonical maps
\[\begin{tikzcd}
    Y \ar[d, "\phi"] \ar[r, dashed, "\Phi_{r, Y}"] & \P\bigl(H^0(Y, rK_Y)\bigr)^\vee \ar[d, "\P(\phi^*)^\vee", "\wr"']\\
    X  \ar[r, dashed, "\Phi_{r, X}"] & \P\bigl(H^0(\Xg, rK_{\Xg})\bigr)^\vee. 
\end{tikzcd}\]
map $X$ and $Y$ birationally to their images. 

For an $r$-form $\alpha \in V_{r, X}$, its pullback $\phi^*\alpha$ lies in $V_{r, Y}$. It is then natural to compare the difference between $\|\alpha\|_s$ and $\|\phi^*\alpha\|_{s}$. But there might be a problem: $Y$ may not be defined over $\O$, so there may not be a canonical measure $\mu_{\Y}$ on $Y(\K)$. Instead, we construct a measure $\mu_Y$ temporary as follows: 
\begin{itemize}
    \item[(\hypertarget{construction}{$\dagger$})] decompose $Y(\K)$ into finitely many polydiscs $Y_j \cong \pi^{k_j} \O^n$, and take $\mu_Y|_{Y_j}$ to be the canonical measure $\mu_\O$ on $\pi^{k_j}\O^n$. Using this measure $\mu_Y$, we define the $s$-norm $\|{-}\|_s$ on $V_{r, Y}$ similarly: 
    \[\|\beta\|_s = \int_{Y(\K)} |\beta|^s \,d\mu_Y^{1-rs}. \]
\end{itemize}

\begin{thr}\label{thrgI}
    For a fixed positive number $s \neq \frac{1}{r}$, the Jacobian $J_\phi\colon Y(\K) \to \R_{\ge 0}$ defined by the formula $\phi^*d\mu_\X = J_\phi\,d\mu_Y$ is determined by the data
    \begin{equation}\label{eqdata}
        \bigl(V_{r, X}, \|{-}\|_{s}\bigr)\quad\text{and}\quad \bigl(V_{r, Y}, \|{-}\|_{s}\bigr). 
    \end{equation}
    In particular, the set of divisors $\mathcal{E} = \{E\}$ and the positive integers $m_E$, $E\in \mathcal{E}$, are also determined by (\ref{eqdata}). 
\end{thr}

\begin{proof}
    It follows from the construction (\hyperlink{construction}{$\dagger$}) of $\mu_Y$ that $J_\phi\colon Y(\K) \to \R_{\ge 0}$ is continuous. Let $\alpha_0$, $\ldots$ , $\alpha_N$ be a basis of $V_{r, X}$, $g_i = \frac{\alpha_i}{\alpha_0}$ a rational function on $Y$ for each $i$, and $d\nu = |\phi^*\alpha_0|^s\,d\mu_Y^{1 - rs}$. Then for a form $\alpha = \alpha_0 + \sum_i \lambda_i g_i$, we have 
    \[\|\alpha\|_s = \int_{Y^\circ(\K)} J_\phi^{1-rs} \left|1 + \sum \lambda_i g_i\right|^s d\nu,\quad \|\phi^*\alpha\|_s = \int_{Y^\circ(\K)} \left|1 + \sum \lambda_i g_i\right|^s d\nu, \]
    where $Y^\circ = Y\setminus \{\phi^*\alpha_0 = 0\}$. 

    For each Borel function $h\colon \K^N \to \R$, denote by $I(h) = (I_X(h), I_Y(h))$ the integrals given by 
    \[I_X(h) = \int_{Y^\circ(\K)} J_\phi^{1-rs}\cdot(h \circ G)\,d\nu,\quad I_Y(h) = \int_{Y^\circ(\K)} (h \circ G)\,d\nu, \]
    where $G = (g_1, \ldots, g_N) \colon Y^\circ \to \textbf{A}^{\!N}$. 
    
    As in the proof of Theorem~\ref{thrbir}, consider the functions $B_s(y)$ and $B_{s,0}(y)$ (defined in (\ref{eqBs}) and (\ref{eqBs0})). Let $A = \alpha_0 + \sum \O\alpha_i$. It follows that
    \[I(B_s) = \bigl(I_X(B_s), I_Y(B_s)\bigr) = \left(\int_{A} \|\alpha\|_s \,d\alpha,\int_{A} \|\phi^*\alpha\|_s \,d\alpha\right), \]
    and hence $I(B_{s,0})$, is determined by (\ref{eqdata}). Using (\ref{eqcharON}), we see that $I(\boldsymbol{1}_{\O^N})$, and hence $I(h)$ for each nonnegative Borel function $h\colon \K^{N} \to \R$, is determined by (\ref{eqdata}). 
    
    Let $U$ be a Zariski open dense subset of $Y^\circ$ on which $G|_U\colon U \to \A^{\!N}$ is an immersion and $G^{-1}(G(U)) = U$. Then for each $y\in U(\K)$, 
    \[J_\phi(y) = \left(\lim_{\varepsilon \to 0^+} \frac{I_X\bigl(\boldsymbol{1}_{B_\varepsilon(G(y))}\bigr)}{I_Y\bigl(\boldsymbol{1}_{B_\varepsilon(G(y))}\bigr)}\right)^{\frac{1}{1-rs}} \]
    is determined by (\ref{eqdata}) (note that $s \neq 1/r$). By the continuity of $J_\phi$, we can then determine the function $J_\phi\colon Y \to \R_{\ge 0}$. 
    
    Thus, the union of the divisors
    \[\bigcup_{E\in \mathcal{E}}E = \bigl\{y\in Y\,\big|\, J_\phi(y) = 0\bigr\} \]
    is determined by (\ref{eqdata}). 
    
    In order to find the order $m_E$ of $E$, we pick a generic point $y\in E$ that does not lies in any other $E'\in \mathcal{E}$. Say $y$ lies in the polydisc $Y_j \cong \pi^{k_j}\O^n$. Then under some coordinate $u = (u_1, \ldots, u_n)$, 
    \[\phi^*d\mu_X = J_\phi(u)\,d\mu_Y = m(u)\cdot |u_1|^{m_E}\,du, \]
    for some nonvanishing continuous function $m(u)$. We may assume that $m(u) \equiv q^{-m}$ is a constant on some polydisc $\Delta$ that contains $y$. Then for $j$ large, 
    \[\left\{u\in \Delta\,\middle|\, J_\phi(u) = q^{-j} \right\} \neq \varnothing \quad \iff \quad j \in m_E \Z + m. \]
    Therefore, $m_E$ can be determined by $J_\phi$, and hence, by (\ref{eqdata}). 
\end{proof}

\begin{cor}\label{corgI}
    Let $\X$ and $\X'$ be smooth projective varieties over $\O$ of relative dimension $n$ such that 
    \[T\colon \bigl(V_{r, X'}, \|{-}\|_s\bigr) \longrightarrow \bigl(V_{r, X}, \|{-}\|_s\bigr) \]
    is an isometry for some positive number $s\neq 1/r$. Suppose that the $r$-canonical maps $\Phi_{r, X}$ and $\Phi_{r, X'}$ maps $X$ and $X'$ birationally to their images. Then there exists a projective smooth variety $Y$ over $\K$ with birational morphisms $\phi\colon Y \to X$ and $\phi'\colon Y \to X'$ such that $\phi^*d\mu_{X}$ is proportional to ${\phi'}^*d\mu_{X'}$. In particular, $\Xg$ and $\Xg'$ are $K$-equivalent on the $\K$-points. 
\end{cor}

\begin{proof}
    It follows from Theorem~\ref{thrbir} that $X$ and $X'$ are birational. So there is a resolution $\phi\colon Y \to X$, such that the birational map $f\colon X \dashrightarrow X'$ factors through some birational morphism $\phi'\colon Y \to X'$: 
    \[\begin{tikzcd}[column sep = tiny]
        & Y \ar[ld, "\phi"'] \ar[rd, "\phi'"] \\
        X \ar[rr, dashed, "f"] && X' 
    \end{tikzcd}\]
    Note that the $r$-canonical map $\Phi_{r, Y}$ also maps $Y$ birationally to its image. 
    
    Write
    \[\phi^*K_{X} = K_Y + \sum m_E E,\quad {\phi'}^* K_{X'} = K_Y + \sum m_E' E. \]
    Let $\mu_Y$ be the measure on $Y(\K)$ we constructed in (\hyperlink{construction}{$\dagger$}), and let $\|{-}\|_s$ be the $s$-norm induced by $\mu_Y$. By Theorem~\ref{thrgI}, the Jacobian $J_\phi = \phi^*d\mu_\X / d\mu_Y$ and $\{(E, m_E)\}_{E\in \mathcal{E}}$ are determined by 
    \[\bigl(V_{r, X}, \|{-}\|_s\bigr)\quad\text{and}\quad \bigl(V_{r, Y}, \|{-}\|_{s}\bigr), \]
    while $J_{\phi'}  = {\phi'}^*d\mu_{\X'} / d\mu_Y$ and $\{(E, m_E')\}_{E\in \mathcal{E}}$ are determined by 
    \[\bigl(V_{r, X'}, \|{-}\|_s\bigr)\quad\text{and}\quad \bigl(V_{r, Y}, \|{-}\|_{s}\bigr). \]
    Since $T = u\cdot f^*$ for some $u \in \K^\times$, we see that 
    \[{\phi'}^*d\mu_{\X'} = J_{\phi'}\,d\mu_{Y} = {|u|}^{\frac{1}{1-rs}}J_\phi\,d\mu_{Y} = {|u|}^{\frac{1}{1-rs}}\,\phi^*d\mu_{\X} \]
    and $\{(E, m_E)\}_{E\in \mathcal{E}} = \{(E, m_E')\}_{E\in \mathcal{E}}$, which shows that $\Xg$ is $K$-equivalent to $\Xg'$ on the $\K$-points. 
\end{proof}

\begin{rmk}\hfill
    \begin{enumerate}[(i)]
        \item In the proof above, if we assume further more that there are resolutions $\phi\colon \Y\to \X$ and $\phi'\colon \Y \to \X'$ over $\O$, then this shows that $\X$ and $\X'$ are $K$-equivalent on the $\K$-points (over $\O$). 
        
        \item For other positive number $s'$, using ${\phi'}^*d\mu_{\X'} = \upsilon\,\phi^*d\mu_{\X}$ for some $\upsilon \in \R_{>0}$, we see that
        \[\|f^*\alpha\|_{s'} = \int \|\phi^*\alpha\|^{s'}J_{\phi'}^{1-rs'}\,d\mu_Y^{1-rs'} = \upsilon^{1-rs'}\int \|\phi^*\alpha\|^{s'}J_{\phi}^{1-rs'}\,d\mu_Y^{1-rs'} = \upsilon^{1-rs'}\|\alpha\|_{s'}, \]
        i.e., under $f^*$, $\|{-}\|_{s'}$ on $V_{r, X'}$ is proportional to $\|{-}\|_{s'}$ on $V_{r, X}$. 
    
        \item When $X$ and $X'$ are already birational to each other, say $f\colon X \dashrightarrow X'$ is a birational map, the order on the $s$-norms also detects the $K$-partial ordering. More precisely, if $s < 1/r$ and 
        \begin{equation}\label{eqnormorder}
            \|f^*\alpha'\|_s \le \|\alpha'\|_s,\quad \forall \alpha' \in H^0(\Xg', rK_{\Xg'}),
        \end{equation}
        then $X \le_{K} X'$, i.e., there exists a birational correspondence (over $\K$)
        \[\begin{tikzcd}[column sep = tiny]
            & Y \ar[ld, "\phi"'] \ar[rd, "\phi'"] \\
            X \ar[rr, dashed] && X' 
        \end{tikzcd}\]
        with $\phi^* K_X \le {\phi'}^* K_{X'}$ on $Y(\K)$. If $s > 1/r$, then the condition (\ref{eqnormorder}) implies $X \ge_{K} X'$. 
        \item The result also applies to the case in which $r$ is a negative integer, for example, when both $X$ and $X'$ are Fano with $-r$ sufficiently large. 
    \end{enumerate}
\end{rmk}

Again, the theorem and the corollary above are also generalized to the log pair case. 

\begin{cor}
    Let $(\X, \D)$ and $(\X', \D')$ be smooth projective klt pairs over $\O$ of relative dimension $n$. Let $L$ (resp.~$L'$) be an $\R$-divisor on $X$ (resp.~$X'$), and let $V$ (resp.~$V'$) be a linear subspace of $V_{r, (X, D), L}$ (resp.~$V_{r, (X', D'), L'}$). 
    
    Suppose that the maps $\Phi_V$ and $\Phi_{V'}$ maps $X$ and $X'$ birationally to their images, and there is an isometry \[T\colon \bigl(V', \|{-}\|_{s, \D'}\bigr) \longrightarrow \bigl(V, \|{-}\|_{s, \D}\bigr) \]
    for some positive number $s < \min \{s_r(X, D, V), s_r(X', D', V')\}$ with $s\neq 1/r$. Then $(X, D)$ and $(X', D')$ are $K$-equivalent on the $\K$-points. 
\end{cor}


\begin{rmk}
    Let $P(\partial_s) = \sum c_k \partial_s^k$ be a linear differential operator of degree $d \ge 1$ where $\partial_s = \tfrac{d}{ds}$. Then 
    \[P(\partial_s)\|\alpha\|_{s, D} = \int_{X(\K)} P(a)t^a\,d\mu_{\X, \D}, \]
    where $a = v(|\alpha| / d\mu_{\X, \D}^r)$. An isometry 
    \[\bigl(V', P(\partial_s)\|{-}\|_{s, \D'}\bigr) \longrightarrow \bigl(V, P(\partial_s)\|{-}\|_{s, \D}\bigr)\] 
    also gives us $K$-equivalence between $(\X, \D)$ and $(\X', \D')$ when they both satisfy (\hyperlink{assumption}{$\spadesuit$}) (even when $s = 0$, $1/r$). 
    
    Indeed, replace $B_s(y)$ in the proof of Theorem~\ref{thrgI} by  
    \[B_s^{P, a}(y) = \int_{\O^N} P\bigl(v\bigl(1 + {\textstyle\sum} \lambda_i y_i\bigr) + a\bigr)\cdot \bigl|1 + {\textstyle\sum} \lambda_i y_i\bigr|^s\,d\lambda. \]
    Then the function
    \[B_{s, 0}^{P, a}(y) \colonequals (-1)^{d + 1} \sum_{j = 0}^{d + 1} \tbinom{d+1}{j} (-t)^{-j} B_s^{P, a}(\pi^j y)\]
    is supported in $\pi^{-d}\O^N$, and one can prove that the integral
    \[Q^P(a) \colonequals \int_{\pi^{-d} \O^N} B_{s, 0}^{P, a}(y)\,dy\]
    is a polynomial in $a$ with $\deg Q^P = d - \delta_{t1}$. Since $Q^P\neq 0$, we can then determine the Jacobian function $J_\phi$ following the proof above. 
\end{rmk}

Applying the same method in Theorem~\ref{thrgI}, we can prove that: 
\begin{pp}
    Let $(\X, \D)$ be a smooth projective klt pair over $\O$, $L$ be an $\R$-divisor on $X$, $V$ be a linear subspace of $V_{r, (X, D), L}$ such that $\Phi_{V}$ maps $X$ birationally to its image. Then the $\Q(t_D)$-valued total norm $\|{-}\|_\D = (\|{-}\|_{s, \D})_{s}$ on $V$ can be determined by any two norms $\|{-}\|_{s_1, \D}$ and $\|{-}\|_{s_2, \D}$ with $s_1\neq s_2$. 
\end{pp}

\begin{proof}
    Suppose that the data $\|{-}\|_{s_1, D}$ and $\|{-}\|_{s_2, D}$ with $s_1 \neq s_2$ are given. Then for each $\alpha \in V$ we can determine 
    \[I_i(h) = \int_{U(\K)}(h\circ G)\cdot|\alpha|^{s_i}\,d\mu_{\X, \D}^{1-rs_i} \]
    for all nonnegative Borel function $h\colon \K^N \to \R$, where $U$ is some Zariski dense open dense subset of $X$ and $G\colon U \to \textbf{A}^{\!N}$ is the immersion defined by $\Phi_{V}$ (up to a choice of basis). 
    
    Thus, we can determine  
    \[\frac{|\alpha|}{d\mu_{\X, \D}^r}(x) = \lim_{\varepsilon \to 0^+}\left(\frac{I_2\bigl(B_\varepsilon(G(x))\bigr)}{I_1\bigl(B_\varepsilon(G(x))\bigr)}\right)^{\frac{1}{s_2 - s_1}}\]
    for each $x\in U(\K)$. Let 
    \[Z = \left\{u \in \K^N \,\middle|\, I_1\bigl(B_\varepsilon(u)) \neq 0,\ \forall \varepsilon > 0\right\}. \]
    It is clear that $Z$ contains $G(U)(\K)$ and is contained in the closure $\overline{G(U)(\K)}$ of $G(U)(\K)$. 
    
    Take 
    \[h(u) = \boldsymbol{1}_Z(u)\cdot\lim_{\varepsilon \to 0^+}\left(\frac{I_2\bigl(B_\varepsilon(u)\bigr)}{I_1\bigl(B_\varepsilon(u)\bigr)}\right)^{\frac{s - s_1}{s_2 - s_1}}, \]
    so $h(G(x)) = (|\alpha|/d\mu_{\X, \D}^r)^{s - s_1}(x)$ for each $x\in U(\K)$. We see that
    \[\|\alpha\|_{s, \D} = \int_{U(\K)}  \left(\frac{|\alpha|}{d\mu_{\X, \D}^r}\right)^{s - s_1}\cdot |\alpha|^{s_1}\,d\mu_{\X, \D}^{1 - rs_1} = I_1(h) \]
    is determined by the norms $\|{-}\|_{s_1, \D}$ and $\|{-}\|_{s_2, \D}$. 
\end{proof}

\newpage

\bibliographystyle{amsalpha}
\bibliography{bib}

\end{document}